\documentclass[a4paper,11pt]{amsart}

\usepackage{amsfonts,amsmath,amssymb,amsthm}
\usepackage[utf8]{inputenc}
\usepackage[T1]{fontenc}
\usepackage[english]{babel}
\usepackage{hyperref}

\newcommand{\R}{\mathbb{R}}
\newcommand{\C}{\mathbb{C}}
\newcommand{\N}{\mathbb{N}}
\newcommand{\Q}{\mathbb{Q}}
\newcommand{\Z}{\mathbb{Z}}
\newcommand{\bs}{\mathbb{S}}
\newcommand{\cA}{\mathcal{A}}
\newcommand{\cZ}{\mathcal{Z}}
\newcommand{\cF}{\mathcal{F}}
\newcommand{\loc}{\mathrm{loc}}
\newcommand{\vol}{\mathrm{vol}}
\newcommand{\pr}{\mathrm{pr}}

\newcommand{\norm}[1]{\left\lVert #1 \right\rVert}
\newcommand{\abs}[1]{\left\lvert #1 \right\rvert}
\newcommand{\ip}[2]{\left\langle #1, #2 \right\rangle}

\theoremstyle{plain}
\newtheorem{thm}{Theorem}[section]
\newtheorem{lem}[thm]{Lemma}
\newtheorem{prop}[thm]{Proposition}
\newtheorem{cor}[thm]{Corollary}

\theoremstyle{definition}

\newtheorem{rem}[thm]{Remark}

\newtheorem*{definition*}{Definition}

\numberwithin{equation}{section}

\begin{document}

\title{Signed quasiregular curves}
\author{Susanna Heikkilä}
\address[]{Department of Mathematics and Statistics, P.O. Box 68 (Pietari Kalmin katu 5), FI-00014 University of Helsinki, Finland}
\email{susanna.a.heikkila@helsinki.fi}
\thanks{This work was supported in part by the Academy of Finland project \#332671}

\begin{abstract}
We define a subclass of quasiregular curves, called signed quasiregular curves, which contains holomorphic curves and quasiregular mappings. As our main result, we prove a growth theorem of Bonk-Heinonen type for signed quasiregular curves. To obtain our main result, we prove that signed quasiregular curves satisfy a weak reverse Hölder inequality and that this weak reverse Hölder inequality implies the main result. We also obtain higher integrability for signed quasiregular curves. Further, we prove a cohomological value distribution result for signed quasiregular curves by using our main result and equidistribution.
\end{abstract}

\maketitle

\section{Introduction}

Our motivation for defining signed quasiregular curves comes from Liouville type growth results in conformal geometry. By the classical Liouville's theorem, every bounded entire function $\C \to \C$ is constant. The same result holds for quasiregular mappings $\R^n \to \R^n$. Recall that a continuous mapping $f\colon M \to N$ between oriented Riemannian $n$-manifolds, $n\geq 2$, is \emph{$K$-quasiregular for $K\geq 1$} if $f$ belongs to the Sobolev space $W^{1,n}_{\loc}(M,N)$ and satisfies the distortion inequality
\[
\norm{Df}^n \leq KJ_f
\]
almost everywhere in $M$, where $\norm{Df}$ is the operator norm of the differential $Df$ of $f$ and $J_f$ is the Jacobian determinant of $f$. We refer to \cite{BI} and \cite{RI} for the theory of quasiregular mappings.

The aforementioned Liouville theorem for quasiregular maps $\R^n \to \R^n$ follows from the following growth bound: \emph{Given $n\geq 2$ and $K\geq 1$ there exists a constant $\varepsilon = \varepsilon(n,K)>0$ so that every $K$-quasiregular mapping $f\colon \R^n \to \R^n$ satisfying
\[
\lim_{\abs{x} \to \infty} \abs{x}^{-\varepsilon}\abs{f(x)}=0
\]
is constant}; see \cite[Corollary III.1.13]{RI}.

For Riemannian manifolds this Euclidean result takes the following form (Bonk and Heinonen \cite[Theorem 1.11]{BH}): \emph{Given $n\geq 2$ and $K\geq 1$ there exists a constant $\varepsilon = \varepsilon(n,K)>0$ so that every nonconstant $K$-quasiregular mapping $f\colon \R^n \to N$ into a closed, connected, and oriented Riemannian $n$-manifold $N$, that is not a rational cohomology sphere, satisfies}
\[
\liminf_{r\to \infty} \frac{1}{r^\varepsilon} \int_{B^n(r)} J_f >0.
\]

In this paper, we prove a version of this Bonk-Heinonen growth result for signed quasiregular curves. For the definition of quasiregular curves, we give the auxiliary definition of an $n$-volume form on an $m$-manifold for $n\leq m$. We say that a smooth differential $n$-form $\omega \in \Omega^n(N)$ on a Riemannian $m$-manifold $N$ is an \emph{$n$-volume form for $n\leq m$} if $\omega$ is closed and pointwise nonvanishing.

A continuous mapping $f\colon M \to N$ between oriented Riemannian manifolds, $2\leq n=\dim M \leq \dim N$, is a \emph{$K$-quasiregular $\omega$-curve for $K\geq 1$ and an $n$-volume form $\omega \in \Omega^n(N)$} if $f$ belongs to the Sobolev space $W^{1,n}_{\loc}(M,N)$ and satisfies the distortion inequality
\[
(\norm{\omega} \circ f)\norm{Df}^n \leq K(\star f^\ast \omega)
\]
almost everywhere in $M$, where $\norm{\omega}$ is the \emph{pointwise comass norm of} $\omega$ given by
\[
\norm{\omega_x} = \max \left\{ \, \abs{\omega_x(v_1,\ldots,v_n)} \colon v_1,\ldots,v_n \text{ unit vectors in } T_xN  \, \right\}
\]
for every $x\in N$ and $(\star f^\ast \omega)$ is the function satisfying $(\star f^\ast \omega)\vol_M = f^\ast \omega$. Quasiregular curves have similar properties as quasiregular mappings to some extent; see \cite{OP2} and \cite{PA2}.

To define the subclass of signed quasiregular curves, we introduce the following terminology.

Let $N$ be a connected and oriented Riemannian $m$-manifold of dimension $m\geq 2$ and let $C_b^\infty(N)$ be the space of all smooth and bounded functions on $N$. For $\ell=1,\ldots,m$, we denote $\cZ_b^\ell(N)$ the space of all smooth, bounded, and closed $\ell$-forms.

For $2\leq n\leq m$, let $\cA_b^n(N)$ be the $C_b^\infty(N)$-algebra generated by products $\alpha \wedge \beta$ of forms $\alpha \in \cZ_b^{\ell}(N)$ and $\beta \in \cZ_b^{n-\ell}(N)$, that is, a smooth $n$-form $\omega \in \Omega^n(N)$ belongs to the space $\cA_b^n(N)$ if there exists $\varphi_i \in C_b^\infty(N)$, $\alpha_i \in \cZ_b^{\ell_i}(N)$, and $\beta_i \in \cZ_b^{n-\ell_i}(N)$ for $i=1,\ldots,j$ for which
\[
\omega= \sum_{i=1}^j \varphi_i \alpha_i \wedge \beta_i.
\]
For example, $\cA_b^n(\R^m)=\Omega_b^n(\R^m)$.

\begin{definition*}
A quasiregular $\omega$-curve $f\colon M \to N$ is \emph{signed with respect to the $n$-volume form $\omega \in \Omega^n(N)$} if $\omega= \sum_{i=1}^j \varphi_i \alpha_i \wedge \beta_i \in \cA_b^n(N)$, where $\varphi_i \in C_b^\infty(N)$, $\alpha_i \in \cZ_b^{\ell_i}(N)$, and $\beta_i \in \cZ_b^{n-\ell_i}(N)$ for $i=1,\ldots,j$, and the almost everywhere defined measurable functions $(\star f^\ast (\alpha_i \wedge \beta_i))$ do not change sign for any $i=1,\ldots,j$. In this case, we also say that the representation $\omega= \sum_{i=1}^j \varphi_i \alpha_i \wedge \beta_i$ is \emph{$f$-signed}.
\end{definition*}

For example, holomorphic curves are signed with respect to the standard symplectic form. In Section \ref{signedRep}, we show that, if $N$ is closed and not a rational cohomology sphere, then $\vol_N$ has a representation in $\cA_b^m(N)$ which is signed with respect to all quasiregular mappings.

We focus on signed quasiregular curves $\R^n \to N$ but our methods work also for quasiregular $\omega$-curves $\R^n \to N$ with respect to an $n$-volume form $\omega$ in the algebra $\cA_b^n(N)$ with constant coefficients. Formally, we say that an $n$-form $\omega= \sum_{i=1}^j \varphi_i \alpha_i \wedge \beta_i \in \cA_b^n(N)$ has an \emph{$\R$-linear representation} if the functions $\varphi_i$ are constant functions for $i=1,\ldots,j$.

We are now ready to state the main result.

\begin{thm}\label{mainResult}
Let $f\colon \R^n \to N$ be a nonconstant $K$-quasiregular $\omega$-curve, where $N$ is a connected and oriented Riemannian $m$-manifold, $m\geq n\geq 2$, $K\geq 1$, and $\omega \in \cA_b^n(N)$ is an $n$-volume form satisfying $\inf_N \norm{\omega} >0$. Suppose that either $\omega$ has an $f$-signed representation or an $\R$-linear representation. Then there exists a constant $\varepsilon = \varepsilon(n,K,\omega)>0$ for which
\begin{equation}\label{fastCondition}
\liminf_{r\to \infty} \frac{1}{r^\varepsilon}\int_{B^n(r)} f^\ast \omega >0.
\end{equation}
\end{thm}

We say that a quasiregular $\omega$-curve $f\colon \R^n \to N$ satisfying the growth condition \eqref{fastCondition} has \emph{fast growth of order} $\varepsilon$. In particular, Theorem \ref{mainResult} yields the following corollary.

\begin{cor}
A nonconstant signed quasiregular $\omega$-curve $f\colon \R^n \to N$ into a closed, connected and oriented Riemannian manifold $N$ with respect to an $n$-volume form $\omega \in \cA_b^n(N)$ has fast growth.
\end{cor}

Theorem \ref{mainResult} yields the following result that may be viewed as a cohomological value distribution result for signed quasiregular curves; see Section \ref{distributionSection}.

\begin{thm}\label{valueResult}
Let $f\colon \R^n \to N$ be a nonconstant $K$-quasiregular $\omega_0$-curve, where $N$ is a connected and oriented Riemannian $m$-manifold, $m\geq n\geq 2$, $K\geq 1$, and $\omega_0 \in \cA_b^n(N)$ is an $n$-volume form satisfying $\inf_N \norm{\omega_0} >0$. Suppose that either $\omega_0$ has an $f$-signed representation or an $\R$-linear representation. Then, for every $\omega \in \Omega^n(N)$ satisfying $\omega_0-\omega=d\tau$ for some $\tau \in \Omega_b^{n-1}(N)$, there exists a set $E\subset [1,\infty)$ for which $\int_E \frac{\mathrm{d}r}{r} < \infty$ and
\[
\liminf_{\substack{r\to \infty \\ r\notin E}} \frac{1}{r^\varepsilon}\int_{B^n(r)} f^\ast \omega >0,
\]
where $\varepsilon=\varepsilon(n,K,\omega_0)>0$. In particular, if $(\star f^\ast \omega) \geq 0$ almost everywhere in $\R^n$, then
\[
\liminf_{r\to \infty} \frac{1}{r^\varepsilon}\int_{B^n(r)} f^\ast \omega >0.
\]
\end{thm}

\begin{rem}
Note that, since $\omega_0 \in \cA_b^n(N)$, the set
\[
\{ \, \omega_0 + d\tau \, \colon \, \tau \in \Omega_b^{n-1}(N) \, \}
\]
contains the bounded de Rham cohomology class of the $n$-volume form $\omega_0 \in \Omega^n(N)$; see e.g. \cite{WI}. In particular, if $N$ is closed, then this set is exactly the de Rham cohomology class of $\omega_0$.
\end{rem}

\begin{cor}
A nonconstant signed quasiregular $\omega_0$-curve $f\colon \R^n \to N$ into a closed, connected and oriented Riemannian manifold $N$ with respect to an $n$-volume form $\omega_0 \in \cA_b^n(N)$ satisfies, for every $\omega \in \Omega^n(N)$ in the de Rham cohomology class of $\omega_0$,
\[
\liminf_{\substack{r\to \infty \\ r\notin E}} \frac{1}{r^\varepsilon}\int_{B^n(r)} f^\ast \omega >0,
\]
where $\varepsilon >0$ and the exception set $E$ has finite logarithmic measure.
\end{cor}

Theorem \ref{mainResult} immediately yields the following corollary for constant coefficient $n$-volume forms. We identify constant coefficient $n$-volume forms in $\R^m$ with $n$-covectors in $\Lambda^n(\R^m)$ in the statement.

\begin{cor}
Let $f\colon \R^n \to \R^m$ be a nonconstant $K$-quasiregular $\omega$-curve, where $m\geq n\geq 2$, $K\geq 1$, and $\omega \in \Lambda^n(\R^m)$ is a nonzero covector. Then $f$ has fast growth of order $\varepsilon=\varepsilon(n,K,\omega)>0$.
\end{cor}

If $\omega$ is an $n$-volume form that splits into $\omega = \omega_1 \wedge \omega_2$, where each $\omega_i$ is a bounded lower dimensional volume form, then every quasiregular $\omega$-curve is signed. Thus, Theorem \ref{mainResult} yields the following corollary.

\begin{cor}
Let $f\colon \R^n \to N$ be a nonconstant $K$-quasiregular $\omega$-curve, where $N$ is a connected and oriented Riemannian $m$-manifold, $m\geq n\geq 2$, $K\geq 1$, and $\omega \in \Omega^n(N)$ is an $n$-volume form satisfying $\inf_N \norm{\omega} >0$. Suppose that $\omega$ splits into $\omega = \omega_1 \wedge \omega_2$, where each $\omega_i \in \Omega_b^{n_i}(N)$ is an $n_i$-volume form for $n_i \geq 1$. Then $f$ has fast growth of order $\varepsilon=\varepsilon(n,K,\omega_1,\omega_2)>0$.
\end{cor}

In particular, we have the following corollary for quasiregular curves into products of closed manifolds.

\begin{cor}
For $i=1,2$, let $N_i$ be a closed, connected and oriented Riemannian $m_i$-manifold, $m_i \geq n_i \geq 1$, and let $\pi_i \colon N_1\times N_2 \to N_i$ be a projection. Let also $n=n_1+n_2$ and let $\omega \in \Omega^n(N_1\times N_2)$ be an $n$-volume form $\omega = \pi_1^\ast \omega_1 \wedge \pi_2^\ast \omega_2$, where each $\omega_i \in \Omega^{n_i}(N_i)$ is an $n_i$-volume form. Then every nonconstant signed quasiregular $\omega$-curve $f\colon \R^n \to N_1\times N_2$ has fast growth.
\end{cor}

Every quasiregular mapping is signed if the target manifold is closed and not a rational cohomology sphere. Hence, we also obtain the growth result of Bonk and Heinonen as a corollary of Theorem \ref{mainResult}.

\begin{cor}\label{BHcorollary}
Let $f\colon \R^n \to N$ be a nonconstant $K$-quasiregular mapping, where $N$ is a closed, connected, and oriented Riemannian $n$-manifold which is not a rational cohomology sphere, $n\geq 2$, and $K\geq 1$. Then $f$ has fast growth of order $\varepsilon=\varepsilon(n,K,\vol_N)>0$.
\end{cor}

\subsection*{Method of the proof of Theorem \ref{mainResult}; weak reverse Hölder inequality}

Our main method of proof for Theorem \ref{mainResult} is a weak reverse Hölder inequality for signed quasiregular curves. The method of proof grew out of the observation that the growth result of Bonk and Heinonen follows almost immediately from the higher integrability result of Prywes \cite[Proposition 2.5]{PR}. The connection between the weak reverse Hölder inequality and fast growth is formally stated in the following theorem.

\begin{thm}\label{secondMainResult}
Let $f\colon \R^n \to N$ be a nonconstant $K$-quasiregular $\omega$-curve, where $N$ is a connected and oriented Riemannian $m$-manifold, $m\geq n\geq 2$, $K\geq 1$, and $\omega \in \Omega^n(N)$ is an $n$-volume form. If there exists constants $p>1$ and $C_p>0$ such that $(\star f^\ast \omega)$ satisfies the weak reverse Hölder inequality
\[
\left( \frac{1}{\abs{B^n\left( \frac{r}{2}\right)}} \int_{B^n\left( \frac{r}{2}\right)} (\star f^\ast \omega)^p \right)^\frac{1}{p} \leq C_p\frac{1}{\abs{B^n(r)}} \int_{B^n(r)} f^\ast \omega
\]
for every $r>0$, then $f$ has fast growth of order $\varepsilon=\varepsilon(n,p)>0$.
\end{thm}

Having Theorem \ref{secondMainResult} at our disposal, it suffices to prove the following version of Prywes' higher integrability result for signed quasiregular curves to obtain Theorem \ref{mainResult}.

\begin{thm}\label{improvedWeakReverseHölderProposition}
Let $f\colon \R^n \to N$ be a nonconstant $K$-quasiregular $\omega$-curve, where $N$ is a connected and oriented Riemannian $m$-manifold, $m\geq n\geq 2$, $K\geq 1$, and $\omega \in \cA_b^n(N)$ is an $n$-volume form satisfying $\inf_N \norm{\omega} >0$. Suppose that either $\omega$ has an $f$-signed representation or an $\R$-linear representation. Then there exists constants $p=p(n,K,\omega)>1$ and $C=C(n,K,\omega)>0$ for which the inequality
\[
\left( \frac{1}{\abs{\frac{1}{2}B}} \int_{\frac{1}{2}B} (\star f^\ast \omega)^p \right)^\frac{1}{p} \leq C\frac{1}{\abs{B}} \int_{B} f^\ast \omega
\]
holds for every open ball $B=B^n(x,r)\subset \R^n$, where $\frac{1}{2}B=B^n\left(x,\frac{r}{2}\right)$.
\end{thm}

As a byproduct of the method, we obtain a higher integrability result similar to \cite[Theorem 1.3]{OP2}.

\begin{thm}\label{higherIntegrability}
Let $f\colon \R^n \to N$ be a nonconstant $K$-quasiregular $\omega$-curve, where $N$ is a connected and oriented Riemannian $m$-manifold, $m\geq n\geq 2$, $K\geq 1$, and $\omega \in \cA_b^n(N)$ is an $n$-volume form satisfying $\inf_N \norm{\omega} >0$. Suppose that either $\omega$ has an $f$-signed representation or an $\R$-linear representation. Then $f\in W_{\loc}^{1,q}(\R^n,N)$ for some $q=q(n,K,\omega)>n$.
\end{thm}

\subsection*{Organization of the article}

In Section \ref{signedRep}, we construct a representation of the Riemannian volume form that is signed with respect to all quasiregular mappings. In Section \ref{growthProof}, we prove Theorem \ref{secondMainResult}. In Section \ref{higherIntegrabilitySection}, we prove Theorems \ref{improvedWeakReverseHölderProposition}, \ref{higherIntegrability}, and \ref{mainResult}. In Section \ref{distributionSection}, we discuss the equidistribution of quasiregular curves and prove Theorem \ref{valueResult}. In Section \ref{ex}, we present a family of examples of signed quasiregular curves.

\subsection*{Acknowledgements}

The author thanks her thesis advisor Pekka Pankka for all his help improving the manuscript.

\section{Signed representation of the Riemannian volume form}\label{signedRep}

In this section, we describe how to construct a representation of the Riemannian volume form that is signed with respect to all quasiregular mappings. The construction is based on the work of Prywes \cite{PR}.

Let $N$ be a closed, connected, and oriented Riemannian $n$-manifold, $n\geq 2$, which is not a rational cohomology sphere. Let $\alpha \in \Omega^\ell(N)$ and $\beta \in \Omega^{n-\ell}(N)$ be closed forms satisfying
\[
\int_N \alpha \wedge \beta = \int_N \vol_N
\]
for some $1\leq \ell \leq n-1$. Following the proof of \cite[Lemma 2.3]{PR}, there exists a positive function $h \in C^\infty(N)$, a smooth partition of unity $\{\lambda_i\}_{i=1}^j$ on $N$, and orientation preserving diffeomorphisms $\Phi_i \colon N\to N$, $i=1,\ldots,j$, for which
\[
\vol_N = h \sum_{i=1}^j \lambda_i \Phi_i^\ast(\alpha \wedge \beta) = \sum_{i=1}^j h\lambda_i (\Phi_i^\ast \alpha \wedge \Phi_i^\ast \beta).
\]
In our terminology, this is a representation of $\vol_N$ in the algebra $\cA_b^n(N)$, since $h\lambda_i \in C_b^\infty(N)$, $\Phi_i^\ast \alpha \in \cZ_b^\ell(N)$, and $\Phi_i^\ast \beta \in \cZ_b^{n-\ell}(N)$ for $i=1,\ldots,j$.

To obtain a representation that is signed with respect to all quasiregular mappings, we apply this construction to a Hodge pair.

\begin{lem}\label{representationLemma}
Let $N$ be a closed, connected, and oriented Riemannian $n$-manifold, $n\geq 2$, which is not a rational cohomology sphere. Then $\vol_N$ has a representation which is signed with respect to all quasiregular mappings into $N$.
\end{lem}
\begin{proof}
Since $N$ is not a rational cohomology sphere, there exists $1\leq \ell \leq n-1$ for which $H^\ell(N)\neq 0$. Let $0\neq c \in H^\ell(N)$ and let $\xi_c \in c$ be the harmonic representative of $c$, that is, $\xi_c$ is the unique form in the de Rham cohomology class $c$ such that $d\xi_c=0$ and $d^\ast \xi_c=0$. We may assume that
\[
\int_N \xi_c \wedge \star \xi_c = \int_N \vol_N.
\]
Then the volume form $\vol_N$ has a representation
\begin{equation}\label{riemexp}
\vol_N = h \sum_{i=1}^j \lambda_i \Phi_i^\ast(\xi_c \wedge \star \xi_c),
\end{equation}
where $h \in C^\infty(N)$ is a positive function, $\{\lambda_i\}_{i=1}^j$ is a smooth partition of unity on $N$, and $\Phi_i \colon N\to N$ is an orientation preserving diffeomorphism for every $i=1,\ldots,j$.

Let $f\colon M \to N$ be a $K$-quasiregular mapping, where $M$ is an oriented Riemannian $n$-manifold and $K\geq 1$. Then
\begin{align*}
f^\ast (\Phi_i^\ast(\xi_c \wedge \star \xi_c)) &= f^\ast (\Phi_i^\ast( \ip{\xi_c}{\xi_c} \vol_N )) = (\ip{\xi_c}{\xi_c} \circ \Phi_i \circ f)(J_{\Phi_i} \circ f)J_f \vol_M \\
&= (\norm{\xi_c}^2 \circ \Phi_i \circ f)(J_{\Phi_i} \circ f)J_f \vol_M
\end{align*}
for every $i=1,\ldots,j$. The function $\norm{\xi_c}^2$ is nonnegative, each Jacobian determinant $J_{\Phi_i}$ is positive, and the Jacobian $J_f$ is nonnegative. Hence, the representation \eqref{riemexp} is $f$-signed.
\end{proof}

\section{Weak reverse Hölder inequality implies fast growth}\label{growthProof}

In this section, we prove Theorem \ref{secondMainResult}.

\begin{proof}[Proof of Theorem \ref{secondMainResult}]
Since
\[
\left( \frac{1}{\abs{B^n\left( \frac{r}{2}\right)}} \int_{B^n\left( \frac{r}{2}\right)} (\star f^\ast \omega)^p \right)^\frac{1}{p} \leq C_p\frac{1}{\abs{B^n(r)}} \int_{B^n(r)} f^\ast \omega
\]
for every $r>0$, we have that
\[
C_p^{-1}\abs{B^n(1)}^{1-\frac{1}{p}}2^{\frac{n}{p}} \left( \int_{B^n\left(\frac{r}{2}\right)} (\star f^\ast \omega)^p \right)^\frac{1}{p} \leq \frac{\int_{B^n(r)} f^\ast \omega}{r^\varepsilon}
\]
for every $r>0$ for $\varepsilon=n\left( 1-\frac{1}{p}\right)>0$. Since $f$ is a nonconstant $K$-quasiregular $\omega$-curve, we may choose $r_0 >0$ for which
\[
\int_{B^n\left(\frac{r_0}{2}\right)} (\star f^\ast \omega)^p >0.
\]
Then, for
\[
C=C_p^{-1}\abs{B^n(1)}^{1-\frac{1}{p}}2^{\frac{n}{p}} >0,
\]
we have the estimate
\[
\frac{\int_{B^n(r)} f^\ast \omega}{r^\varepsilon} \geq C\left( \int_{B^n\left(\frac{r}{2}\right)} (\star f^\ast \omega)^p \right)^\frac{1}{p} \geq C\left( \int_{B^n\left(\frac{r_0}{2}\right)} (\star f^\ast \omega)^p \right)^\frac{1}{p}
\]
for every $r\geq r_0$. Thus
\[
\liminf_{r\to \infty} \frac{\int_{B^n(r)} f^\ast \omega}{r^\varepsilon} \geq C\left( \int_{B^n\left(\frac{r_0}{2}\right)} (\star f^\ast \omega)^p \right)^\frac{1}{p}.
\]
This proves the theorem.
\end{proof}

\section{Higher integrability of signed quasiregular curves}\label{higherIntegrabilitySection}

Let $f\colon \R^n \to N$ be a quasiregular $\omega$-curve. Then $u=(\star f^\ast \omega)$ is a locally integrable function and we may apply known results for locally integrable functions such as Gehring's lemma; see e.g. \cite[Theorem 4.2]{BI}. Hence, we may prove Theorem \ref{improvedWeakReverseHölderProposition} by proving the following result that is similar to results for quasiregular mappings due to Prywes (see \cite[Proposition 2.2]{PR} and \cite[Lemma 2.4]{PR}). The proof uses techniques developed there.

\begin{prop}\label{weakReverseHölderProposition}
Let $f\colon \R^n \to N$ be a nonconstant $K$-quasiregular $\omega$-curve, where $N$ is a connected and oriented Riemannian $m$-manifold, $m\geq n\geq 2$, $K\geq 1$, and $\omega \in \cA_b^n(N)$ is an $n$-volume form satisfying $\inf_N \norm{\omega} >0$. Suppose that either $\omega$ has an $f$-signed representation or an $\R$-linear representation. Then
\[
\frac{1}{\abs{\frac{1}{2}B}}\int_{\frac{1}{2}B} f^\ast \omega \leq C(n,K,\omega) \left( \frac{1}{\abs{B}}\int_{B} (\star f^\ast \omega)^\frac{n}{n+1} \right)^\frac{n+1}{n}
\]
for every open ball $B\subset \R^n$.
\end{prop}

\begin{proof}
Let 
\[
\cF(\omega) = \inf \sum_{i=1}^j \norm{\varphi_i}_\infty \norm{\alpha_i}_\infty \norm{\beta_i}_\infty,
\]
where the infimum is taken over all $f$-signed and $\R$-linear representations $\omega = \sum_{i=1}^j \varphi_i \alpha_i \wedge \beta_i$ of $\omega$ in $\cA_b^n(N)$. Then $0<\cF(\omega)<\infty$. Thus, there exists either an $f$-signed or an $\R$-linear representation $\omega = \sum_{i=1}^j \varphi_i \alpha_i \wedge \beta_i$ for which
\[
\sum_{i=1}^j \norm{\varphi_i}_\infty \norm{\alpha_i}_\infty \norm{\beta_i}_\infty < 2\cF(\omega).
\]

For every $i=1,\ldots,j$, the forms $f^\ast \alpha_i$ and $f^\ast \beta_i$ are weakly closed since the forms $\alpha_i$ and $\beta_i$ are closed. By writing $\alpha_i \wedge \beta_i = ((-1)^{\ell_i(n-\ell_i)} \beta_i) \wedge \alpha_i$ if necessary, we may assume that for every $i=1,\ldots,j$, $\alpha_i \in \cZ_b^{\ell_i}(N)$ for $1\leq \ell_i \leq \frac{n}{2}$.

Let $B\subset \R^n$ be an open ball. Then, by the Poincaré inequality for differential forms, there exists for every $i=1,\ldots,j$ a differential form $\tau_i \in W^{1,p_i}(\Lambda^{\ell_i-1}B)$ satisfying $d\tau_i=f^\ast \alpha_i$ in the weak sense and
\begin{equation}\label{spineq}
\norm{\tau_i}_{\frac{nq_i}{n-q_i},B} \leq C(n)\norm{f^\ast \alpha_i}_{q_i,B},
\end{equation}
where $p_i=\frac{n}{\ell_i}$ and $q_i=\frac{n}{\ell_i}\frac{n}{n+1}$; see Iwaniec and Lutoborski \cite[Corollary 4.2]{IL}.

Let $\psi \in C_c^\infty(\R^n)$ be a nonnegative function such that $\psi(x)=1$ for every $x\in \frac{1}{2}B$, $\psi(x)=0$ for every $x\in \R^n \setminus B$, and $\abs{d\psi} \leq \frac{3}{r}$. By the nonnegativity of $\psi(\star f^\ast \omega)$, we have that
\[
\int_{\frac{1}{2}B} f^\ast \omega \leq \int_B \psi f^\ast \omega = \sum_{i=1}^j \int_B \psi (\varphi_i \circ f) f^\ast (\alpha_i \wedge \beta_i).
\]

Suppose first that the representation $\omega = \sum_{i=1}^j \varphi_i \alpha_i \wedge \beta_i$ is $\R$-linear. Then the functions $\varphi_i$ are constant functions $\varphi_i \equiv c_i$ for every $i=1,\ldots,j$ and
\begin{align*}
\sum_{i=1}^j \int_B \psi (\varphi_i \circ f) f^\ast (\alpha_i \wedge \beta_i) &= \sum_{i=1}^j c_i \int_B \psi f^\ast (\alpha_i \wedge \beta_i) \\
&\leq \sum_{i=1}^j \norm{\varphi_i}_\infty \abs{\int_B \psi f^\ast (\alpha_i \wedge \beta_i)}.
\end{align*}

Suppose now that the representation $\omega = \sum_{i=1}^j \varphi_i \alpha_i \wedge \beta_i$ is $f$-signed. For every $i=1,\ldots,j$, the function $(\star f^\ast (\alpha_i \wedge \beta_i))$ is either nonnegative or nonpositive almost everywhere. If the function $(\star f^\ast (\alpha_i \wedge \beta_i))$ is nonnegative, then
\begin{align*}
\int_B \psi (\varphi_i \circ f) f^\ast (\alpha_i \wedge \beta_i) &\leq \norm{\varphi_i}_\infty \int_B \psi f^\ast (\alpha_i \wedge \beta_i) \\
&= \norm{\varphi_i}_\infty \abs{\int_B \psi f^\ast (\alpha_i \wedge \beta_i)}.
\end{align*}
On the other hand, if the function $(\star f^\ast (\alpha_i \wedge \beta_i))$ is nonpositive, then
\begin{align*}
\int_B \psi (\varphi_i \circ f) f^\ast (\alpha_i \wedge \beta_i) &= \int_B \psi (-\varphi_i \circ f) (-f^\ast (\alpha_i \wedge \beta_i)) \\
&\leq \norm{\varphi_i}_\infty \int_B \psi (-f^\ast (\alpha_i \wedge \beta_i)) \\
&= \norm{\varphi_i}_\infty \abs{\int_B \psi f^\ast (\alpha_i \wedge \beta_i)}.
\end{align*}
Thus
\[
\sum_{i=1}^j \int_B \psi (\varphi_i \circ f) f^\ast (\alpha_i \wedge \beta_i) \leq \sum_{i=1}^j \norm{\varphi_i}_\infty \abs{\int_B \psi f^\ast (\alpha_i \wedge \beta_i)}.
\]

Hence, it suffices to estimate the term
\[
\abs{\int_B \psi f^\ast (\alpha_i \wedge \beta_i)} = \abs{\int_B d\tau_i \wedge (\psi f^\ast \beta_i)}
\]
for every $i=1,\ldots,j$. We obtain by integration by parts that
\[
\abs{\int_B d\tau_i \wedge (\psi f^\ast \beta_i)} = \abs{\int_B \tau_i \wedge d\psi \wedge f^\ast \beta_i} \leq C(n)\int_B \abs{\tau_i} \abs{d\psi} \abs{f^\ast \beta_i}
\]
for every $i=1,\ldots,j$. Further, by the inequality $\abs{d\psi} \leq \frac{3}{r}$, Hölder's inequality and inequality \eqref{spineq}, we have the estimate
\begin{align*}
C(n)\int_B \abs{\tau_i} \abs{d\psi} \abs{f^\ast \beta_i} &\leq \frac{C(n)}{r} \norm{\tau_i}_{\frac{nq_i}{n-q_i},B} \norm{f^\ast \beta_i}_{\frac{n}{n-\ell_i}\frac{n}{n+1},B} \\
&\leq \frac{C(n)}{r} \norm{f^\ast \alpha_i}_{q_i,B} \norm{f^\ast \beta_i}_{\frac{n}{n-\ell_i}\frac{n}{n+1},B}
\end{align*}
for every $i=1,\ldots,j$.

Since $f$ is a $K$-quasiregular $\omega$-curve, we may estimate each term $\norm{f^\ast \alpha_i}_{q_i,B}$ by
\begin{align*}
\norm{f^\ast \alpha_i}_{q_i,B} &= \left( \int_B \abs{f^\ast \alpha_i}^{\frac{n}{\ell_i}\frac{n}{n+1}} \right)^{\frac{\ell_i}{n}\frac{n+1}{n}} \leq C(n)\norm{\alpha_i}_\infty \left( \int_B \norm{Df}^{n\frac{n}{n+1}} \right)^{\frac{\ell_i}{n}\frac{n+1}{n}}  \\
&\leq C(n)\norm{\alpha_i}_\infty \left( \inf_N \norm{\omega} \right)^{-\frac{\ell_i}{n}} \left( \int_B ((\norm{\omega} \circ f) \norm{Df}^n)^{\frac{n}{n+1}} \right)^{\frac{\ell_i}{n}\frac{n+1}{n}} \\
&\leq C(n)\norm{\alpha_i}_\infty \left( \inf_N \norm{\omega} \right)^{-\frac{\ell_i}{n}} K^\frac{\ell_i}{n} \left( \int_B (\star f^\ast \omega)^\frac{n}{n+1} \right)^{\frac{\ell_i}{n}\frac{n+1}{n}}.
\end{align*}
Similarly, we get the estimate
\begin{align*}
&\norm{f^\ast \beta_i}_{\frac{n}{n-\ell_i}\frac{n}{n+1},B} \\
&\quad \leq C(n)\norm{\beta_i}_\infty \left( \inf_N \norm{\omega} \right)^\frac{\ell_i-n}{n} K^\frac{n-\ell_i}{n} \left( \int_B (\star f^\ast \omega)^\frac{n}{n+1} \right)^{\frac{n-\ell_i}{n}\frac{n+1}{n}}
\end{align*}
for every $i=1,\ldots,j$.

Combining all the estimates, we arrive at
\begin{align*}
&\frac{1}{\abs{\frac{1}{2}B}} \int_{\frac{1}{2}B} f^\ast \omega \\
&\quad \leq \left( \sum_{i=1}^j \norm{\varphi_i}_\infty \norm{\alpha_i}_\infty \norm{\beta_i}_\infty \right) \left( \inf_N \norm{\omega} \right)^{-1} \frac{C(n)}{r^{n+1}} K \left( \int_B (\star f^\ast \omega)^\frac{n}{n+1} \right)^\frac{n+1}{n} \\
&\quad < 2\cF(\omega) C(\omega) C(n) K \left( \frac{1}{r^n} \int_B (\star f^\ast \omega)^\frac{n}{n+1} \right)^\frac{n+1}{n} \\
&\quad = C(n,K,\omega) \left( \frac{1}{\abs{B}} \int_B (\star f^\ast \omega)^\frac{n}{n+1} \right)^\frac{n+1}{n}.
\end{align*}
This concludes the proof.
\end{proof}

As mentioned, Theorem \ref{improvedWeakReverseHölderProposition} now follows from Gehring's lemma and Proposition \ref{weakReverseHölderProposition}. Consequently, Theorem \ref{higherIntegrability} is almost immediate.

\begin{proof}[Proof of Theorem \ref{higherIntegrability}]
By Theorem \ref{improvedWeakReverseHölderProposition}, $(\star f^\ast \omega) \in L_{\loc}^p(\R^n)$ for some $p=p(n,K,\omega)>1$. Let $q=np>n$. Then, for every bounded domain $U\subset \R^n$, we have the estimate
\begin{align*}
\int_U \norm{Df}^q &\leq \left( \inf_N \norm{\omega} \right)^{-p} \int_U ((\norm{\omega} \circ f)\norm{Df}^n)^p \\
&\leq \left( \inf_N \norm{\omega} \right)^{-p} K^p \int_U (\star f^\ast \omega)^p <\infty.
\end{align*}
Hence, $\norm{Df} \in L_{\loc}^q(\R^n)$ and $f\in W_{\loc}^{1,q}(\R^n,N)$.
\end{proof}

Next, we combine Theorems \ref{improvedWeakReverseHölderProposition} and \ref{secondMainResult} to prove Theorem \ref{mainResult}.

\begin{proof}[Proof of Theorem \ref{mainResult}]
We have by Theorem \ref{improvedWeakReverseHölderProposition} that there exists constants $p=p(n,K,\omega)>1$ and $C=C(n,K,\omega)>0$ for which the inequality
\[
\left( \frac{1}{\abs{\frac{1}{2}B}} \int_{\frac{1}{2}B} (\star f^\ast \omega)^p \right)^\frac{1}{p} \leq C\frac{1}{\abs{B}} \int_{B} f^\ast \omega
\]
holds for every open ball $B\subset \R^n$. Then, by considering open balls $B^n(r)\subset \R^n$ for $r>0$, the claim follows from Theorem \ref{secondMainResult}.
\end{proof}

\section{Equidistribution of quasiregular curves}\label{distributionSection}

We begin by proving an equidistribution theorem for quasiregular curves in the spirit of Mattila and Rickman \cite{MR}. The following statement is analogous to result of Pankka \cite[Theorem 4]{PA3} and while the proof is essentially the same, we present it for the reader's convenience.

\begin{thm}\label{equidistribution}
Let $f\colon \R^n \to N$ be a nonconstant $K$-quasiregular $\omega_0$-curve, where $N$ is a connected and oriented Riemannian $m$-manifold, $m\geq n\geq 2$, $K\geq 1$, and $\omega_0 \in \Omega^n(N)$ is an $n$-volume form satisfying $\inf_N \norm{\omega_0} >0$. Suppose that the function $A_{\omega_0,f} \colon (0,\infty) \to [0,\infty)$, $r\mapsto \int_{B^n(r)} f^\ast \omega_0$, is unbounded. Then, for every $\omega \in \Omega^n(N)$ satisfying $\omega_0-\omega=d\tau$ for some $\tau \in \Omega_b^{n-1}(N)$, there exists a set $E\subset [1,\infty)$ for which $\int_E \frac{\mathrm{d}r}{r} < \infty$ and
\[
\lim_{\substack{r\to \infty \\ r\notin E}} \frac{\int_{B^n(r)} f^\ast \omega}{\int_{B^n(r)} f^\ast \omega_0} = 1.
\]
\end{thm}

\begin{rem}
Theorem \ref{equidistribution} is cohomological contrary to the result in \cite{MR} which is for measures. This important distinction is illustrated further in Section \ref{ex}. There, we give an example of a volume form $\omega$ for which there exists both a quasiregular $\omega$-curve with nowhere dense image and a quasiregular $\omega$-curve with dense image.
\end{rem}

\begin{proof}[Proof of Theorem \ref{equidistribution}]
Let $\omega \in \Omega^n(N)$ and $\tau \in \Omega_b^{n-1}(N)$ be forms such that $\omega_0-\omega=d\tau$. Let $\frac{n-1}{n}<\delta<1$ and let $E\subset [1,\infty)$ be the set of all radii $r\geq 1$ satisfying
\[
\left( \int_{B^n(r)} f^\ast \omega_0 \right)^\delta \leq \abs{\int_{\bs^{n-1}(r)} f^\ast \tau}.
\]
Then Hölder's inequality yields the estimate
\begin{align*}
&\left( \int_{B^n(r)} f^\ast \omega_0 \right)^\delta \\
&\quad \leq \abs{\int_{\bs^{n-1}(r)} f^\ast \tau} \leq \int_{\bs^{n-1}(r)} \abs{f^\ast \tau} \leq C(n) \norm{\tau}_\infty \int_{\bs^{n-1}(r)} \norm{Df}^{n-1} \\
&\quad \leq C(n) \norm{\tau}_\infty \left( \int_{\bs^{n-1}(r)} \norm{Df}^n \right)^\frac{n-1}{n} \left( \int_{\bs^{n-1}(r)} 1 \right)^\frac{1}{n} \\
&\quad \leq C(n) \norm{\tau}_\infty r^\frac{n-1}{n} \left( \inf_N \norm{\omega_0} \right)^\frac{1-n}{n} \left( \int_{\bs^{n-1}(r)} (\norm{\omega_0} \circ f)\norm{Df}^n \right)^\frac{n-1}{n} \\
&\quad \leq C(n) \norm{\tau}_\infty r^\frac{n-1}{n} \left( \inf_N \norm{\omega_0} \right)^\frac{1-n}{n} K^\frac{n-1}{n} \left( \int_{\bs^{n-1}(r)} f^\ast \omega_0 \right)^\frac{n-1}{n}
\end{align*}
for almost every $r \in E$.

Since $(\star f^\ast \omega_0) \in L_{\loc}^1(\R^n)$, the function $A_{\omega_0,f}$ is differentiable almost everywhere and
\[
(A_{\omega_0,f})'(r) = \int_{\bs^{n-1}(r)} f^\ast \omega_0
\]
for almost every $r>0$; see e.g. \cite[Theorem 3.12]{EG} for details. Thus, by the previous estimate, the function $A_{\omega_0,f}$ satisfies the differential inequality
\[
(A_{\omega_0,f}(r))^{\delta\frac{n}{n-1}} \leq C r (A_{\omega_0,f})'(r)
\]
for almost every $r \in E$, where $C=C(n,\tau,\omega_0,K)$. Since $A_{\omega_0,f}$ is unbounded, we may choose $r_0 >0$ for which $A_{\omega_0,f}(r_0)>0$. Then, for every $R>r_0$, the function $\frac{1}{1-\delta\frac{n}{n-1}}(A_{\omega_0,f})^{1-\delta\frac{n}{n-1}}$ is well-defined and nondecreasing on the interval $[r_0,R]$. Hence
\begin{align*}
\int_E \frac{\mathrm{d}r}{r} &\leq \log(r_0) + \lim_{R\to \infty} C\int_{r_0}^R \frac{(A_{\omega_0,f})'(r)}{(A_{\omega_0,f}(r))^{\delta\frac{n}{n-1}}} \, \mathrm{d}r \\
&\leq \log(r_0) + \lim_{R\to \infty} \frac{C}{1-\delta\frac{n}{n-1}} \left((A_{\omega_0,f}(R))^{1-\delta\frac{n}{n-1}} - (A_{\omega_0,f}(r_0))^{1-\delta\frac{n}{n-1}} \right) \\
&= \log(r_0) + \frac{C}{\delta\frac{n}{n-1}-1} (A_{\omega_0,f}(r_0))^{1-\delta\frac{n}{n-1}} < \infty;
\end{align*}
see \cite[Theorem 18.14]{HS} for details for the second step.

Finally, since $\delta <1$, we have that
\begin{align*}
\abs{\frac{\int_{B^n(r)} f^\ast \omega}{\int_{B^n(r)} f^\ast \omega_0} -1} &= \abs{\frac{\int_{B^n(r)} f^\ast (\omega_0 - d\tau)}{\int_{B^n(r)} f^\ast \omega_0} -1} = \abs{\frac{\int_{B^n(r)} f^\ast d\tau}{\int_{B^n(r)} f^\ast \omega_0}} \\
&= \frac{\abs{\int_{\bs^{n-1}(r)} f^\ast \tau}}{\int_{B^n(r)} f^\ast \omega_0} \leq \frac{\left( \int_{B^n(r)} f^\ast \omega_0 \right)^\delta}{\int_{B^n(r)} f^\ast \omega_0} = (A_{\omega_0,f}(r))^{\delta-1} \to 0
\end{align*}
as $r\to \infty$, $r\notin E$. This concludes the proof.
\end{proof}

Having Theorem \ref{equidistribution} at our disposal, we have that, if a quasiregular $\omega_0$-curve has fast growth, then the growth condition is also satisfied in a large set of radii for any form $\omega$ satisfying $\omega_0-\omega=d\tau$ for a bounded $\tau$.

\begin{prop}\label{boundedCohomologyDependence}
Let $f\colon \R^n \to N$ be a nonconstant $K$-quasiregular $\omega_0$-curve, where $N$ is a connected and oriented Riemannian $m$-manifold, $m\geq n\geq 2$, $K\geq 1$, and $\omega_0 \in \Omega^n(N)$ is an $n$-volume form satisfying $\inf_N \norm{\omega_0} >0$. Suppose that $f$ has fast growth of order $\varepsilon >0$. Then, for every $\omega \in \Omega^n(N)$ satisfying $\omega_0-\omega=d\tau$ for some $\tau \in \Omega_b^{n-1}(N)$, there exists a set $E\subset [1,\infty)$ for which $\int_E \frac{\mathrm{d}r}{r} < \infty$ and
\[
\liminf_{\substack{r\to \infty \\ r\notin E}} \frac{1}{r^\varepsilon}\int_{B^n(r)} f^\ast \omega >0.
\]
In particular, if $(\star f^\ast \omega) \geq 0$ almost everywhere in $\R^n$, then
\[
\liminf_{r\to \infty} \frac{1}{r^\varepsilon}\int_{B^n(r)} f^\ast \omega >0.
\]
\end{prop}

\begin{rem}
In the proof of Proposition \ref{boundedCohomologyDependence}, we use the following property of sets of finite logarithmic measure: \emph{Given $E\subset [1,\infty)$ for which $\int_E \frac{\mathrm{d}r}{r} < \infty$ there exists $i_0 \in \N$ having the property that for every $i\geq i_0$ and every $t\in E\cap [2^i,2^{i+1}]$, $\left[ \frac{t}{2},t \right] \not\subset E$.} This property follows immediately from the observation that $\int_\frac{t}{2}^t \frac{\mathrm{d}r}{r} = \ln 2$.
\end{rem}

\begin{proof}[Proof of Proposition \ref{boundedCohomologyDependence}]
Let $\omega \in \Omega^n(N)$ and $\tau \in \Omega_b^{n-1}(N)$ be differential forms satisfying $\omega_0-\omega=d\tau$.  The function $A_{\omega_0,f} \colon (0,\infty) \to [0,\infty)$, $r\mapsto \int_{B^n(r)} f^\ast \omega_0$, is unbounded since $f$ has fast growth of order $\varepsilon$. Hence, by Theorem \ref{equidistribution}, there exists a set $E\subset [1,\infty)$ satisfying $\int_E \frac{\mathrm{d}r}{r} < \infty$ and
\[
\lim_{\substack{r\to \infty \\ r\notin E}} \frac{\int_{B^n(r)} f^\ast \omega}{\int_{B^n(r)} f^\ast \omega_0} = 1.
\]
Then
\begin{align*}
\liminf_{\substack{r\to \infty \\ r\notin E}} \frac{1}{r^\varepsilon}\int_{B^n(r)} f^\ast \omega &= \liminf_{\substack{r\to \infty \\ r\notin E}} \frac{\int_{B^n(r)} f^\ast \omega_0}{r^\varepsilon} \frac{\int_{B^n(r)} f^\ast \omega}{\int_{B^n(r)} f^\ast \omega_0} \\
&\geq \left( \liminf_{\substack{r\to \infty \\ r\notin E}} \frac{\int_{B^n(r)} f^\ast \omega_0}{r^\varepsilon} \right) \left( \liminf_{\substack{r\to \infty \\ r\notin E}} \frac{\int_{B^n(r)} f^\ast \omega}{\int_{B^n(r)} f^\ast \omega_0} \right) \\
&\geq \left( \liminf_{r\to \infty} \frac{1}{r^\varepsilon}\int_{B^n(r)} f^\ast \omega_0 \right) >0.
\end{align*}

Now suppose that $(\star f^\ast \omega) \geq 0$ almost everywhere in $\R^n$. Suppose towards contradiction that
\[
\liminf_{r\to \infty} \frac{\int_{B^n(r)} f^\ast \omega}{r^\varepsilon} =0.
\]
For every $k\in \N$, we may choose $r_k \geq 2^k$ for which
\[
\inf_{s\geq r_k} \frac{\int_{B^n(s)} f^\ast \omega}{s^\varepsilon} <\frac{1}{k}.
\]
Hence
\[
\frac{\int_{B^n(s_k)} f^\ast \omega}{s_k^\varepsilon} <\frac{1}{k}
\]
for some $s_k \geq r_k$. In particular,
\[
\lim_{k\to \infty} \frac{\int_{B^n(s_k)} f^\ast \omega}{s_k^\varepsilon} =0.
\]

We have two cases. Suppose first that there exists $k_0 \in \N$ so that $s_k \notin E$ for every $k\geq k_0$. Then
\[
0= \lim_{k\to \infty} \frac{\int_{B^n(s_k)} f^\ast \omega}{s_k^\varepsilon} \geq \liminf_{\substack{r\to \infty \\ r\notin E}} \frac{\int_{B^n(r)} f^\ast \omega}{r^\varepsilon} >0,
\]
which is a contradiction. Suppose now that for every $j\in \N$ there exists $k_j \geq j$ so that $s_{k_j} \in E$. Since $s_{k_j} \geq r_{k_j} \geq 2^{k_j}$ for every $j\in \N$, there exists $j_0 \in \N$ so that for every $j\geq j_0$ there exists $t_j \notin E$ satisfying $\frac{s_{k_j}}{2} \leq t_j < s_{k_j}$. Then
\[
\frac{\int_{B^n(t_j)} f^\ast \omega}{t_j^\varepsilon} \leq 2^\varepsilon \frac{\int_{B^n(s_{k_j})} f^\ast \omega}{s_{k_j}^\varepsilon} < \frac{2^\varepsilon}{k_j}
\]
for every $j\geq j_0$. Hence
\[
\lim_{j\to \infty} \frac{\int_{B^n(t_j)} f^\ast \omega}{t_j^\varepsilon} =0.
\]
Again, this leads to a contradiction, since
\[
0= \lim_{j\to \infty} \frac{\int_{B^n(t_j)} f^\ast \omega}{t_j^\varepsilon} \geq \liminf_{\substack{r\to \infty \\ r\notin E}} \frac{\int_{B^n(r)} f^\ast \omega}{r^\varepsilon} >0.
\]
\end{proof}

We are now ready to prove Theorem \ref{valueResult}.

\begin{proof}[Proof of Theorem \ref{valueResult}]
Let $\omega \in \Omega^n(N)$ and $\tau \in \Omega_b^{n-1}(N)$ be differential forms satisfying $\omega_0-\omega=d\tau$. By Theorem \ref{mainResult}, $f$ has fast growth of order $\varepsilon=\varepsilon(n,K,\omega_0)>0$. Then, by Proposition \ref{boundedCohomologyDependence}, there exists a set $E\subset [1,\infty)$ for which $\int_E \frac{\mathrm{d}r}{r} < \infty$ and
\[
\liminf_{\substack{r\to \infty \\ r\notin E}} \frac{1}{r^\varepsilon}\int_{B^n(r)} f^\ast \omega >0.
\]
In particular, if $(\star f^\ast \omega) \geq 0$ almost everywhere in $\R^n$, then
\[
\liminf_{r\to \infty} \frac{1}{r^\varepsilon}\int_{B^n(r)} f^\ast \omega >0.
\]
This completes the proof.
\end{proof}

\section{A family of examples}\label{ex}

In this section, we give an example of an $n$-volume form which admits both quasiregular curves with dense image and quasiregular curves with nowhere dense image. This complements the equidistribution result in Section \ref{distributionSection} (Theorem \ref{equidistribution}).

Let $n\geq 2$ and let $T^{n+1} = \bs^1 \times \dots \times \bs^1$ be the $(n+1)$-dimensional torus equipped with the product Riemannian metric. Let $\pr_j \colon T^{n+1} \to \bs^1$ be the $j$th projection $(e^{2\pi ix_1},\ldots,e^{2\pi ix_{n+1}}) \mapsto e^{2\pi ix_j}$ for every $j=1,\ldots,n$ and let $\pi \colon \R^{n+1} \to T^{n+1}$ be the standard locally isometric covering map $(x_1,\ldots,x_{n+1}) \mapsto (e^{2\pi ix_1},\ldots,e^{2\pi ix_{n+1}})$. Let $\vol_{\bs^1} \in \Omega^1(\bs^1)$ be the standard angular form and let
\[
\omega = \pr_1^\ast \vol_{\bs^1} \wedge \dots \wedge \pr_n^\ast \vol_{\bs^1} \in \Omega^n(T^{n+1}).
\]
Clearly, $\omega$ is an $n$-volume form on $T^{n+1}$ satisfying $\norm{\omega}=1$ and $\pi^\ast \omega = dx_1 \wedge \dots \wedge dx_n \in \Omega^n(\R^{n+1})$. Further, $\omega$ has an $\R$-linear representation in the algebra $\cA_b^n(T^{n+1})$ and every quasiregular $\omega$-curve is signed with respect to $\omega$.

Let $y=(y_1,\ldots,y_n)\in \R^n$ and let $L_y \colon \R^n \to \R^{n+1}$ be the mapping $x\mapsto (x,x\cdot y)$. Let also $f_y \colon \R^n \to T^{n+1}$ be the composed mapping $\pi \circ L_y \in C^\infty(\R^n,T^{n+1})$. Since
\[
\norm{Df_y}^n = \norm{D(\pi \circ L_y)}^n = \norm{((D\pi) \circ L_y)DL_y}^n = \norm{DL_y}^n \leq (1 + \abs{y})^n
\]
and
\[
f_y^\ast \omega = (\pi \circ L_y)^\ast \omega = L_y^\ast (dx_1 \wedge \dots \wedge dx_n) = \vol_{\R^n},
\]
the mapping $f_y$ is a $(1 + \abs{y})^n$-quasiregular $\omega$-curve.

Next we show that the image $f_y(\R^n)\subset T^{n+1}$ is nowhere dense if $y\in \Q^n$ and dense if $y\in (\R \setminus \Q)^n$. To that end, denote the Riemannian distance function on $\bs^1$ by $d_1$ and the Riemannian distance function on $T^{n+1}$ by $d_{n+1}$.

\subsection*{Rational case}

Suppose that $y\in \Q^n$. Recall that for a rational angle $\theta$ the set $\{\, e^{2\pi ik\theta} \colon k\in \Z \,\} \subset \bs^1$ is finite.

Suppose towards contradiction that the closure of $f_y(\R^n)$ has nonempty interior. Then there exists a point $v=(v_1,\ldots,v_{n+1})\in \Q^n \times (\R \setminus \Q)$ for which $\pi(v)=(e^{2\pi iv_1},\ldots,e^{2\pi iv_{n+1}})\in T^{n+1}$ belongs to the interior of the closure of $f_y(\R^n)$. For every $j=1,\ldots,n$, we have $v_j=\frac{p_j}{q_j}$ for some $p_j \in \Z$ and $q_j \in \Z_+$. Let
\[
E = \left\{ \, e^{2\pi ik_1 \frac{y_1}{q_1}} \cdot \dots \cdot e^{2\pi ik_n \frac{y_n}{q_n}} \colon k_1,\ldots,k_n \in \Z \, \right\} \subset \bs^1.
\]
The set $E$ is finite and $e^{2\pi iv_{n+1}} \notin E$. Let
\[
r = \frac{\min \{ \, d_1(z,e^{2\pi iv_{n+1}}) \colon z\in E \, \}}{4} >0.
\]
Now for each $t\in \R$ either $d_1(e^{2\pi it},e^{2\pi iv_{n+1}})>r$ or $d_1(e^{2\pi it},E)>r$. We show that this leads to a contradiction.

Let $0<\delta<r$ be such that $n\delta \left( \max_{1\leq j\leq n} \abs{y_j} \right)<r$. Since $\pi(v)$ belongs to the interior of the closure of $f_y(\R^n)$, there exists a point $x\in \R^n$ satisfying $d_{n+1}(f_y(x),\pi(v))<\delta$. Denote
\[
U = \left( v_1-\frac{1}{2},v_1+\frac{1}{2} \right) \times \dots \times \left( v_{n+1}-\frac{1}{2},v_{n+1}+\frac{1}{2} \right) \subset \R^{n+1}.
\]
Let $\ell=(\ell_1,\ldots,\ell_{n+1})\in \Z^{n+1}$ be so that $L_y(x)-\ell \in U$. Then
\begin{align*}
d_1(e^{2\pi ix\cdot y},e^{2\pi iv_{n+1}}) &= \abs{x\cdot y-\ell_{n+1}-v_{n+1}} \leq \abs{L_y(x)-\ell-v} \\
&= d_{n+1}(f_v(x),\pi(v))<\delta<r.
\end{align*}
Also
\begin{align*}
&d_1(e^{2\pi ix\cdot y},e^{2\pi i(q_1\ell_1+p_1)\frac{y_1}{q_1}} \cdot \dots \cdot e^{2\pi i(q_n\ell_n+p_n)\frac{y_n}{q_n}}) \\
&\quad \leq \abs{x\cdot y - \sum_{j=1}^n (q_j\ell_j+p_j)\frac{y_j}{q_j}} = \abs{\sum_{j=1}^n x_jy_j - \sum_{j=1}^n \ell_jy_j + \frac{p_j}{q_j}y_j} \\
&\quad = \abs{\sum_{j=1}^n (x_j-\ell_j-v_j)y_j} \leq \left( \max_{1\leq j\leq n} \abs{y_j} \right) \sum_{j=1}^n \abs{x_j-\ell_j-v_j} \\
&\quad \leq \left( \max_{1\leq j\leq n} \abs{y_j} \right) \sum_{j=1}^n \abs{L_y(x)-\ell-v} < \left( \max_{1\leq j\leq n} \abs{y_j} \right) n\delta <r.
\end{align*}
Since $e^{2\pi i(q_1\ell_1+p_1)\frac{y_1}{q_1}} \cdot \dots \cdot e^{2\pi i(q_n\ell_n+p_n)\frac{y_n}{q_n}}\in E$, we arrive at a contradiction.

\subsection*{Irrational case}

Suppose that $y\in (\R \setminus \Q)^n$. Recall that for an irrational angle $\theta$ the set $\{\, e^{2\pi ik\theta} \colon k\in \Z \,\}$ is dense in $\bs^1$.

Let $v=(v_1,\ldots,v_n)\in \R^{n+1}$. It suffices to show that for every $\delta >0$ there exists $x\in \R^n$ for which $d_{n+1}(f_y(x),\pi(v))<\delta$.

Let $\delta>0$. The set $\{\, e^{2\pi iky_n} \colon k\in \Z \,\}$ is dense in $\bs^1$, so we may choose $k_n\in \Z$ for which $d_1(e^{2\pi i(-v_ny_n+v_{n+1})},e^{2\pi ik_ny_n})<\delta/n$. Let $h_n \in \Z$ be such that $-v_ny_n+v_{n+1}-h_n \in (k_ny_n-1/2, k_ny_n+1/2)$. Similarly, for every $j=1,\ldots,n-1$, there exists $k_j,h_j \in \Z$ satisfying $d_1(e^{2\pi i(-v_jy_j)},e^{2\pi ik_jy_j})<\delta/n$ and $-v_jy_j-h_j \in (k_jy_j-1/2,k_jy_j+1/2)$. Let $x=(v_1+k_1,\ldots,v_n+k_n)\in \R^n$ and $k=(k_1,\ldots,k_n,-\sum_{j=1}^n h_j)\in \Z^{n+1}$. Then
\[
L_y(x)-k-v = \left( 0,\ldots,0,\sum_{j=1}^n (v_j+k_j)y_j+\sum_{j=1}^n h_j-v_{n+1} \right)
\]
and
\begin{align*}
d_{n+1}(f_y(x),\pi(v)) &= d_{n+1}(\pi(L_y(x)-k),\pi(v)) \leq \abs{L_y(x)-k-v} \\
&= \abs{\sum_{j=1}^n (v_j+k_j)y_j + \sum_{j=1}^n h_j - v_{n+1}} \\
&\leq \abs{v_ny_n-v_{n+1}+h_n+k_ny_n} + \sum_{j=1}^{n-1} \abs{v_jy_j+h_j+k_jy_j}.
\end{align*}
Since $\abs{v_ny_n-v_{n+1}+h_n+k_ny_n}=d_1(e^{2\pi i(-v_ny_n+v_{n+1})},e^{2\pi ik_ny_n})<\delta/n$ and $\abs{v_jy_j+h_j+k_jy_j}=d_1(e^{2\pi i(-v_jy_j)},e^{2\pi ik_jy_j})<\delta/n$ for $j=1,\ldots,n-1$, we have that $d_{n+1}(f_y(x),\pi(v))<\delta$.

\end{document}